\documentclass[a4paper]{amsart}
\usepackage{graphicx}
\usepackage{amssymb}
\usepackage{amsmath}
\usepackage{amsthm,amsfonts,bbm}
\usepackage{amscd}
\usepackage[all,2cell]{xy}

\UseAllTwocells \SilentMatrices

\newtheorem{thm}{Theorem}[section]

\newtheorem{cor}[thm]{Corollary}
\newtheorem{lem}[thm]{Lemma}

\newtheorem{prop}[thm]{Proposition}
\newtheorem{prop-def}[thm]{Proposition-Definition}
\theoremstyle{definition}
\newtheorem{Def}[thm]{Definition}
\theoremstyle{remark}
\newtheorem{rmk}[thm]{\bf Remark}
\newtheorem{exm}[thm]{\bf Example}
\numberwithin{equation}{section}
\numberwithin{figure}{section}

\newcommand{\A}{\mathcal{A}}

\def\la{\lambda}

\def\A{\mathcal{A}}

\def\S{\mathcal{S}}
\def\PS{\mathbb{S}}

\def\diag{{\rm diag}}

\def\x{{\mathbf x}}
\def\y{{\mathbf y}}
\def\z{{\mathbf z}}
\def\v{{\mathbf v}}
\def\C{\mathbb{C}}

\def\Z{\mathbb{Z}}

\def\V{\mathcal{V}}
\def \PV{\mathbb{V}}

\def\I{\mathcal{I}}
\def\B{\mathcal{B}}

\def\diag{{\rm diag}}
\def \i{\mathbf{i}}
\def \A{\mathcal{A}}

\def \Dg{\mathfrak{D}}
\def \D{\mathcal{D}}
\def \Spec{\mbox{\rm Spec}}
\def\rank{\mbox{\rm rank}}
\def\sp{\mbox{\rm supp}}
\def\cl{\mbox{\rm cl}}
\def\pc{\mbox{\rm pc}}

\begin{document}
\title[Eigenvariety of Nonnegative Symmetric Weakly Irreducible Tensors]
{Eigenvariety of Nonnegative Symmetric Weakly Irreducible Tensors Associated with Spectral Radius and Its Application to Hypergraphs}

\author[Y.-Z. Fan]{Yi-Zheng Fan$^*$}
\address{School of Mathematical Sciences, Anhui University, Hefei 230601, P. R. China}
\email{fanyz@ahu.edu.cn}
\thanks{$^*$The corresponding author.
The first and the third authors were supported by National Natural Science Foundation of China \#11371028.
The second author was supported by National Natural Science Foundation of China \#11401001.}

\author[Y.-H. Bao]{Yan-Hong Bao}
\address{School of Mathematical Sciences, Anhui University, Hefei 230601, P. R. China}
\email{baoyh@ahu.edu.cn}

\author[T. Huang]{Tao Huang}
\address{School of Mathematical Sciences, Anhui University, Hefei 230601, P. R. China}
\email{huangtao@ahu.edu.cn}

\date{\today}





\begin{abstract}
For a nonnegative symmetric weakly irreducible tensor, 
  its spectral radius is an eigenvalue corresponding to a unique positive eigenvector up to a scalar called the Perron vector.
But including the Perron vector, there may have more than one eigenvector corresponding to the spectral radius.
The projective eigenvariety associated with the spectral radius is the set of the eigenvectors corresponding to the spectral radius considered in the complex projective space.

In this paper we proved that such projective eigenvariety admits a module structure,
   which is determined by the support of the tensor and can be characterized  explicitly by solving the Smith normal form of the incidence matrix of the tensor.
We introduced two parameters: the stabilizing index and the stabilizing dimension of the tensor, where the former is exactly the cardinality of the projective eigenvariety
  and the latter is the composition length of the projective eigenvariety as a module.
We give some upper bounds for the two parameters, and characterize the case that there is only one eigenvector of the tensor corresponding to the spectral radius, i.e. the Perron vector.
By applying the above results to the adjacency tensor of a connected uniform hypergraph,
 we give some upper bounds for the two parameters in terms of the structural parameters of the hypergraph such as path cover number, matching number and the maximum length of paths.
\end{abstract}

\maketitle

\section{Introduction}
A \emph{tensor} (also called \emph{hypermatrix}) $\A=(a_{i_{1} i_2 \ldots i_{m}})$ of order $m$ and dimension $n$ over a field $\mathbb{F}$ refers to a
 multiarray of entries $a_{i_{1}i_2\ldots i_{m}}\in \mathbb{F}$ for all $i_{j}\in [n]:=\{1,2,\ldots,n\}$ and $j\in [m]$.
 In this paper, we mainly consider complex tensors, i.e. $\mathbb{F}=\C$.
If all entries $a_{i_1i_2\cdots i_m}$ of $\A$ are invariant under any permutation of its indices, then $\A$ is called a \emph{symmetric tensor}.

For a given a vector $\x=(x_1, \cdots, x_n)\in \C^n$, $\A\x^{m-1}\in \C^n$, which is defined as
\begin{align*}
(\A\x^{m-1})_i=\sum_{i_2,\cdots, i_m\in [n]}a_{ii_2\cdots i_m}x_{i_2}\cdots x_{i_m}, i\in [n].
\end{align*}
Let $\I=(i_{i_1i_2\cdots i_m})$ be the \emph{identity tensor} of order $m$ and dimension $n$, that is,
$i_{i_1i_2\cdots i_m}=1$ for $i_1=i_2=\cdots=i_m$ and $i_{i_1i_2\cdots i_m}=0$ otherwise.
An \emph{eigenvalue} $\la\in \C$ of $\A$ means that the polynomial system $(\la \I-\A)\x^{m-1}=0$,
  or equivalently $\A\x^{m-1}=\la \x^{[m-1]}$, has a nontrivial solution $\x \in \C^n$ which is called an \emph{eigenvector} of $\A$ corresponding to $\la$, where $\x^{[m-1]}:=(x_1^{m-1}, \cdots, x_n^{m-1})$; see \cite{Lim,Qi1}.

The \emph{determinant} of $\A$, denoted by $\det \A$, is defined as the resultant of the polynomials $\A \x^{m-1}$,
and the \emph{characteristic polynomial} $\varphi_\A(\la)$ of $\A$ is defined as $\det(\la \I-\A)$ \cite{CPZ2,Qi1}.
It is known that $\la$ is an eigenvalue of $\A$ if and only if it is a root of $\varphi_\A(\la)$.
The \emph{algebraic multiplicity} of $\la$ as an eigenvalue of $\A$  is defined to be the multiplicity of $\la$ as a root of $\varphi_\A(\la)$.
The \emph{spectrum} of $\A$, denoted by $\Spec(\A)$, is the multi-set of the roots of $\varphi_\A(\la)$.
The largest modulus of the elements in $\Spec(\A)$ is called the \emph{spectral radius} of $\A$, denoted by $\rho(\A)$.

Let $\la$ be an eigenvalue of $\A$ and $\V_\la=\V_\la(\A)$ the set of all eigenvectors of $\A$ associated with $\la$ together with zero, i.e.
$$\V_\la=\V_\la(\A)=\{\x \in \mathbb{C}^n: \A\x^{m-1}=\la \x^{[m-1]}\}.$$
Observe that the system $(\la \I-\A)\x^{m-1}=0$ is not linear yet for $m\ge 3$,
and therefore $\V_\la$ is not a linear subspace of $\mathbb{C}^n$ in general.
In fact, $\V_\la$ forms an affine variety in $\C^n$ \cite{Ha}, which is called the \emph{eigenvariety}
of $\A$ associated with $\la$ \cite{HuYe}.
Let $\mathbb{P}^{n-1}$ be the standard complex projective spaces of dimension $n-1$.
Since each polynomial in the system $(\la \I-\A)\x^{m-1}=0$ is homogenous of degree $m-1$,
we consider the projective variety
$$\PV_\la=\PV_\la(\A)=\{\x \in \mathbb{P}^{n-1}: \A\x^{m-1}=\la \x^{[m-1]}\}.$$
which is called the \emph{projective eigenvariety} of $\A$ associated with $\la$.

By Perron-Frobenius theorem, it is known that if $\A$ is a nonnegative irreducible matrix,
then $\rho(\A)$ is a simple eigenvalue of $\A$ corresponding a positive eigenvector (called {\it Perron vector}).
So $\PV_{\rho(\A)}$ contains only one element.
But, in general for $\A$ being a nonnegative irreducible or weakly irreducible tensor,
including the Perron vector, $\A$ may have more than one eigenvector corresponding to $\rho(\A)$ up to a scalar,
i.e. $\PV_{\rho(\A)}$ may contains more than one element \cite{FHB}.
So,\emph{it is a natural problem to characterize the dimension or the cardinality of $\PV_{\rho(\A)}$.}

In this paper we focus on the projective eigenvariety $\PV_{\rho(\A)}$ of a nonnegative combinatorial symmetric weakly irreducible tensor $\A$ of order $m$ and dimension $n$.
We define a quasi-Hadamard product $\circ$ in $\PV_{\rho(\A)}$ such that $(\PV_{\rho(\A)}, \circ)$ admits a $\Z_m$-module.
Then we have a $\Z_m$-module isomorphism from $\PV_{\rho(\A)}$ to $\PS_0(\A)$,
where $$\PS_0(\A)=\{\x \in \Z_m^n: B_\A \x=0, x_1=0\},$$ and $B_\A$ is the incidence matrix of $\A$.
By solving the Smith normal form of $B_\A$,
  we can get an explicit decomposition of the $\Z_m$-module $\PS_0(\A)$  into the direct sum of cyclic submodules by elementary linear transformations.
Note that $B_\A$, and hence $\PS_0(\A)$ is only determined by the support or the zero-nonzero pattern of $\A$.
So the algebraic structure of $\PV_{\rho(\A)}$ including its cardinality is determined by the support of $\A$, which can be get by the Smith normal form of $B_\A$.
This answers the above problem partially, that is,
for a nonnegative combinatorial symmetric weakly irreducible tensor $\A$, the dimension of $\PV_{\rho(\A)}$ is zero, i.e. there are a finite number of eigenvectors of $\A$  corresponding to $\rho(\A)$ up to a scalar.

To investigate this problem, we need to introduce a group associated with a general tensor $\A$ of order $m$ and dimension $n$.
Let $\ell$ be a positive integer.
For $j=0,1,\ldots,\ell-1$, define
\begin{equation} \label{Dg}
\begin{split}
\Dg^{(j)}&=\{D: \A=e^{-\i \frac{2\pi j}{\ell}}D^{-(m-1)}\A D, d_{11}=1\},\\
\Dg &=\cup_{j=0}^{\ell-1}\Dg^{(j)},
\end{split}
\end{equation}
where  $D$ is an $n \times n$ invertible diagonal matrix in above definition.
Sometimes we use $\Dg(\A)$ and $\Dg^{(j)}(\A)$ to avoid confusions.
By Lemma \ref{group}, $\Dg$ is an abelian group containing $\Dg^{(0)}$ as a subgroup.
Furthermore, the groups  $\Dg$ and $\Dg^{(0)}$ are only determined by the support of $\A$ by Lemma \ref{comb}.
So, they are both combinatorial invariants of $\A$.

Suppose $\A$ is further nonnegative and weakly irreducible.
Then $\Dg$ and $\Dg^{(0)}$ are closely related to the projective eigenvarieties $\PV_{\la_j}$, where
\begin{equation}\label{laj}
\la_j=\rho(\A)e^{-\i \frac{2\pi j}{\ell}}, j=0,1,\ldots,\ell-1.
\end{equation}
(Note that if $\Dg^{(j)} \ne \emptyset$, then $\la_j$ is an eigenvalue of $\A$.)
The fact is as follows.
By Lemma \ref{ev} for any $\y \in \PV_{\la_j}$, $|\y|$ is the unique positive Perron vector of $\A$ up to a scalar.
Define
\begin{equation}\label{Dy}
D_\y=\diag\left(\frac{y_1}{|y_1|},\cdots,\frac{y_n}{|y_n|}\right).
\end{equation}
Then $\A=e^{-\i  \frac{2\pi j}{\ell}}D_\y^{-(m-1)}\A D_\y$ by Theorem \ref{PF2}(2).
So, if normalizing $\y$ such that $y_1=1$, then $D_\y \in \Dg^{(j)}$.
In fact, $\Dg^{(j)}$ consists of those $D_\y$ arising from $\y \in \PV_{\la_j}$ by Lemma \ref{group}(2).

Let $\mathbf{S}=\{e^{\i \frac{2 \pi j}{\ell}}\A: j=0,1,\ldots,\ell-1\}$.
Suppose that $\Dg^{(1)} \ne \emptyset$.
Then $\Dg$ acts on $\mathbf{S}$ as a permutation group by means of $\A^D:=D^{-(m-1)}\A D$,
where $\Dg^{(0)}$ acts as a stabilizer (or ``self-rotation'') of $\A$,
  and the quotient group $\Dg/\Dg^{(0)}$ acts as a ``rotation'' of $\A$ over $\mathbf{S}$.
The group $\Dg^{(0)}$ reflects the ``\emph{self-rotating behavior}'' of $\A$, and $\Dg/\Dg^{(0)}$ reflects the spectral symmetry of $\A$ because
$\Spec(\A)=e^{\i \frac{2\pi}{\ell}}\Spec(\A)$ from its generator $\Dg^{(1)}$.
So we have two important parameters of $\A$ based on $\Dg$ and $\Dg^{(0)}$.

\begin{Def}\cite{FHB}\label{ell-sym}
Let $\A$ be an $m$-th order $n$-dimensional tensor, and let $\ell$ be a positive integer.
The tensor $\A$ is called \emph{spectral $\ell$-symmetric} if
\begin{equation}\label{sym-For}\Spec(\A)=e^{\i \frac{2\pi}{\ell}}\Spec(\A).\end{equation}
The maximum number $\ell$ such that (\ref{sym-For}) holds is called the \emph{cyclic index} of $\A$ and denoted by $c(\A)$ \cite{CPZ2}.
\end{Def}

\begin{Def}\label{stab}
Let $\A$ be an $m$-th order $n$-dimensional tensor.
The \emph{stabilizing index} of $\A$, denoted by $s(\A)$, is defined as the cardinality of the group $\Dg^{(0)}(\A)$.
\end{Def}

By the above definitions, if $\A$ is nonnegative weakly irreducible,
then $\A$ is spectral $\ell$-symmetric if and only if $\Dg^{(1)}\ne \emptyset$ by Theorem \ref{PF2}(2), and $\ell$ is exactly the cardinality of the quotient $\Dg/\Dg^{(0)}$.
Furthermore, the cyclic index $c(\A)$ is exactly the number of distinct eigenvalues of $\A$ with modulus $\rho(\A)$ by Theorem \ref{specsymm},
the stabilizing index $s(\A)$ is exactly the cardinality of $\PV_{\rho(\A)}$ by Lemma \ref{group}(2), or the number of eigenvectors of $\A$  corresponding to $\rho(\A)$ up to a scalar.

Under the actions of $\Dg$ and $\Dg^{(0)}$,
Fan et.al \cite{FHB} get some structural properties of nonnegative weakly irreducible tensors similar to those of nonnegative irreducible matrices.
The cyclic index $c(\A)$ of a nonnegative weakly irreducible tensor $\A$ was given explicitly in \cite{FHB} by using the generalized traces.
In particular, if $\A$ is a symmetric tensor of order $m$, then $c(\A)|m$.
However, we know little about the stabilizing index $s(\A)$, even in the case  of $\A$ being symmetric or combinatorial symmetric,
which will be addressed in this paper.

The paper is organized as follows.
In Section 2 we introduce some basic knowledge including the properties of the groups of $\Dg$ and $\Dg^{(0)}$.
In Sections 3 we show that $\PV_{\rho(\A)}$ admits a $\Z_m$-module which is isomorphism to $\PS_0(\A)$,
and give an explicit decomposition of the $\Z_m$-module $\PS_0(\A)$ as well as $\PV_{\rho(\A)}$ by using the Smith normal form of $B_\A$.
We also give some upper bounds for the stabilizing index and stabilizing dimension of $\A$, where
the latter is defined to be the composition length of the $\Z_m$-module $\PV_{\rho(\A)}$.
Finally in Section 4 we apply those results to the adjacency tensor of hypergraphs and bound from the above the parameters by some structural parameters of the hypergraphs.

\section{Preliminaries}

\subsection{Perron-Frobenius theorem for nonnegative tensors}
Let $\A=(a_{i_{1}i_2\ldots i_{m}})$ be a tensor of order $m$ and dimension $n$.
Then $\A$ is called \emph{reducible} if there exists a nonempty proper index subset $I \subset [n]$ such that
$a_{i_{1}i_2\ldots i_{m}}=0$ for any $i_1 \in I$ and any $i_2,\ldots,i_m \notin I$;
if $\A$ is not reducible, then it is called \emph{irreducible} \cite{CPZ}.
We also can associate $\A$ with a directed graph $D(\A)$ on vertex set $[n]$ such that $(i,j)$ is an arc of $D(\A)$ if
and only if there exists a nonzero entry $ a_{ii_2\ldots i_{m}}$ such that $j \in \{ i_2\ldots i_{m}\}$.
Then $\A$ is called \emph{weakly irreducible} if $D(\A)$ is strongly connected; otherwise it is called \emph{weakly reducible} \cite{FGH}.
It is known that if $\A$ is irreducible, then it is weakly irreducible; but the converse is not true.

Chang et.al \cite{CPZ} generalize the Perron-Frobenius theorem for nonnegative matrices to nonnegative tensors.
Yang and Yang \cite{YY1,YY2,YY3} get further results for Perron-Frobenius theorem, especially for the spectral symmetry.
Friedland et.al \cite{FGH} also get some results for weakly irreducible nonnegative tensors.
We combine those results in the following theorem, where an eigenvalue is called \emph{$H^+$-eigenvalue}
(respectively, \emph{$H^{++}$-eigenvalue}) if it is corresponding to a nonnegative (respectively, positive) eigenvector.

\begin{thm}[The Perron-Frobenius theorem for nonnegative tensors]\label{PF1}~~
\begin{enumerate}
\item{\em(Yang and Yang \cite{YY1})}  If $\A$ is a nonnegative tensor of order $m$ and dimension $n$, then $\rho(\A)$ is an $H^+$-eigenvalue of $\A$.

\item{\em(Friedland, Gaubert and Han \cite{FGH})} If furthermore $\A$ is weakly irreducible, then $\rho(\A)$ is the unique $H^{++}$-eigenvalue of $\A$,
with the unique positive eigenvector, up to a positive scalar.

\item{\em(Chang, Pearson and Zhang \cite{CPZ})} If moreover $\A$ is irreducible, then $\rho(\A)$ is the unique $H^{+}$-eigenvalue of $\A$,
with the unique nonnegative eigenvector, up to a positive scalar.
\end{enumerate}
\end{thm}

\begin{lem}\cite{YY3} \label{ev}
Let $\A$ be a weakly irreducible nonnegative tensor.
Let $\y$ be an eigenvector of $\A$ corresponding to an eigenvalue $\la$ with $|\la|=\rho(\A)$.
Then $|\y|$ is the unique positive eigenvector corresponding to $\rho(\A)$ up to a scalar.
\end{lem}

According to the definition of tensor product in \cite{Shao}, for a tensor $\A$ of order $m$ and dimension $n$, and two diagonal matrices $P,Q$ both of dimension $n$,
the product $P\A Q$ has the same order and dimension as $\A$, whose entries are given by
\[(P\A Q)_{i_1i_2\ldots i_m}=p_{i_1i_1}a_{i_1i_2\ldots i_m}q_{i_2i_2}\ldots q_{i_mi_m}.\]
If $P=Q^{-1}$, then $\A$ and $P^{m-1}\A Q$ are called \emph{diagonal similar}.
It is proved that two diagonal similar tensors have the same spectrum \cite{Shao}.

\begin{thm}\cite{YY3}\label{PF2}
Let $\A$ and $\B$ be $m$-th order $n$-dimensional tensors with $|\B|\le \A$. Then
\begin{enumerate}

\item[(1)] $\rho(\B)\le \rho(\A)$.

\item[(2)] If $\A$ is weakly irreducible and $\rho(\B)=\rho(\A)$, where $\la=\rho(\A)e^{\i\theta}$ is
an eigenvalue of $\B$ corresponding to an eigenvector $\y$, then $\y=(y_1, \cdots, y_n)$ contains no zero entries, and $\A=e^{-\i\theta}D^{-(m-1)}\B D$,
where $D=\diag(\frac{y_1}{|y_1|}, \cdots, \frac{y_n}{|y_n|})$.
\end{enumerate}
\end{thm}

\begin{thm}\cite{YY3}\label{specsymm}
Let $\A$ be an $m$-th order $n$-dimensional nonnegative weakly irreducible tensor.
Suppose that $\A$ has $k$ distinct eigenvalues with modulus $\rho(\A)$ in total. Then
these eigenvalues are $\rho(\A)e^{\i\frac{2\pi j}{k}}$, $j=0, 1, \cdots, k-1$.
Furthermore, there exists a diagonal matrix $D$ with unit diagonal entries such that
\[\A=e^{-\i\frac{2\pi}{k}}D^{-(m-1)}\A D,\]
and the spectrum of $\A$ keeps invariant under a rotation of angle $\frac{2\pi}{k}$ (but not a smaller positive angle) of the complex plane.
\end{thm}

Lemma \ref{ev}, Theorem \ref{PF2} and Theorem \ref{specsymm} were given by Yang and Yang \cite{YY3} and posted in arXiv.
For completeness and for convenience to the reads, we rewrite their proofs in the Appendix of this paper.

\subsection{The groups associated with tensors}
The {\it support} of $\A$, also called the {\it zero-nonzero pattern} of $\A$ in \cite{Shao}, denoted by $\sp(\A)=(s_{i_1i_2\ldots i_m})$, is defined as a tensor with same order and dimension as $\A$, such that $s_{i_1\ldots i_m}=1$ if $a_{i_1\ldots i_m} \ne 0$,
and $s_{i_1\ldots i_m}=0$ otherwise.
Let $\B$ be a tensor with same order and dimension as $\A$.
We say $\B$ is a \emph{subpattern} of $\A$ if $\sp(\B) \le \sp(\A)$.
A tensor is called {\it combinatorial symmetric} if its support is symmetric.
The part of the following lemma was shown by Fan et.al \cite{FHB} and posted in arXiv.
For completeness we rewrite some of its proof here.

\begin{lem}\cite{FHB} \label{group}
Let $\A$ be an $m$-th order $n$-dimensional nonnegative weakly irreducible tensor, which is spectral $\ell$-symmetric.
Let $\la_j=\rho(\A)e^{\i \frac{2 \pi j}{\ell}}$ be the eigenvalue of $\A$, $j=0,1,\ldots,\ell-1$.
Let $\Dg$ and $\Dg^{(j)}$ be as defined in (\ref{Dg}) for $j=0,1,\ldots,\ell-1$.
Then the following results hold.

\begin{enumerate}

\item $\Dg$ is an abelian group under the usual matrix multiplication, where $\Dg^{(0)}$ is a subgroup of $\Dg$, and
$\Dg^{(j)}$ is a coset of $\Dg^{(0)}$ in $\Dg$ for $j \in [\ell-1]$.

\item there is a bijection between $\PV_{\la_j}(\A)$ and $\Dg^{(j)}(\A)$, and
\begin{equation} \label{2nd}
\Dg^{(j)}(\A)=\{D_\y: \y \in \PV_{\la_j}(\A), y_1=1\}, j=0,1,\ldots,\ell-1.
\end{equation}

\item If further $\A$ is combinatorial symmetric, then $\ell\mid m$, and $D^{m}=\I$ for any $D \in \Dg$.
\end{enumerate}
\end{lem}

\begin{proof}
(1) Surely the identity matrix $\I \in \Dg^{(0)}$. For any two matrices $D^{(j_1)} \in \Dg^{(j_1)}$ and $D^{(j_2)} \in \Dg^{(j_2)}$, we have
\[\A = e^{-\i \frac{2\pi j_1}{\ell}} {D^{(j_1)}}^{-(m-1)}\A D^{(j_1)},
    \A=e^{-\i \frac{2\pi j_2}{\ell}} {D^{(j_2)}}^{-(m-1)}\A D^{(j_2)}.\]
Then \[ \A = e^{-\i \frac{2\pi (j_1+j_2)}{\ell}} {(D^{(j_1)}D^{(j_2)})}^{-(m-1)}\A (D^{(j_1)}D^{(j_2)}).\]
So $D^{(j_1)}D^{(j_2)} \in \Dg^{(j_1+j_2)}$, where the superscript is taken modulo $\ell$.
It is seen that $(D^{(j_1)})^{-1}\in \Dg^{(-j_1)}$.
So $\Dg$ is an abelian group under the usual matrix multiplication.

Following the same routine, one can verify that $\Dg^{(0)}$ is a subgroup of $\Dg$.
Taking a $D^{(j)} \in \Dg^{(j)}$, one can show that $\Dg^{(j)}=\Dg^{(0)} D^{(j)}$, i.e. $\Dg^{(j)}$ is a coset of $\Dg^{(0)}$
by verifying $\bar{D}^{(j)}{D^{(j)}}^{-1} \in \Dg^{(0)}$ and $D^{(0)}{D^{(j)}} \in \Dg^{(j)}$ for any $\bar{D}^{(j)} \in \Dg^{(j)}$ and $D^{(0)} \in \Dg^{(0)}$.

(2) For any $\y \in \PV_{\la_j}(\A)$, by Lemma \ref{ev}, $|\y|>0$.
Without loss of generality, assume $y_1=1$.
Define $D_{\y}$ as in (\ref{Dy}).
Then $D_{\y} \in \Dg^{(j)}(\A)$ by Theorem \ref{PF2}(2).
So we get a injective map $$\psi: \PV_{\la_j}(\A) \to \Dg^{(j)}(\A) \;(\y \mapsto D_{\y}).$$
On the other hand, for any $D \in \Dg^{(j)}(\A)$, we have $\A = e^{-\i \frac{2 \pi j}{\ell}}D^{-(m-1)}\A D$.
Note that by Theorem \ref{PF1}(2), there exists a positive eigenvector $\x$ with $x_1=1$ such that $\A \x^{m-1}=\la_0 \x^{[m-1]}$.
So we have $(e^{-\i \frac{2 \pi j}{\ell}}D^{-(m-1)}\A D) \x^{m-1}=\la_0 \x^{[m-1]}$, and hence $$\A (D\x)^{m-1}=\la_j (D\x)^{[m-1]}.$$
which implies that $D\x \in \PV_{\la_j}(\A)$ with $(D\x)_1=1$.
As $\psi(D\x)=D$, $\psi$ is also surjective.

(3)  Suppose that $\A$ is combinatorial symmetric.
From $\A =e^{-\i \frac{2\pi j}{\ell}}  D^{-(m-1)}\A D$, letting $d_{ii}=e^{\i \theta_i}$ for $i \in [n]$ where $\theta_1=0$,
  if $a_{i_1 \ldots i_m} \ne 0$, we have
 \begin{equation}\label{ell-j-eq}\frac{2\pi j}{\ell}+m\theta_{i_1} \equiv \theta_{i_1}+\cdots+\theta_{i_m} \mod 2\pi.\end{equation}
  As $\A$ is combinatorial symmetric, replacing $i_1$ in left side of (\ref{ell-j-eq}) by $i_l$ and summing over all $l=1,\ldots,m$, we have
 \[\frac{2\pi jm}{\ell}+m\sum_{l=1}^m\theta_{i_l}\equiv m\sum_{l=1}^m\theta_{i_l} \mod 2\pi.\]
  So, $\frac{2\pi jm}{\ell}\equiv 0 \mod 2\pi$.
  If taking $j=1$, we have $\ell\mid m$.
  Also from (\ref{ell-j-eq}), we have
\[m\theta_{i_1} = \ldots=m\theta_{i_m} \mod 2\pi.\]
  As $\A$ is weakly irreducible and $\theta_1=0$,
  we get $m\theta_i \equiv 0 \mod 2\pi.$
  So $D^m=\I$ for any $D \in \Dg$.
\end{proof}

We next show that the group $\Dg(\A)$ and the subgroup $\Dg^{(0)}(\A)$ are combinatorial invariants,
that is, they are only determined by the support of $\A$.

\begin{lem}\label{comb}
Let $\A,\B$ be two tensors both with order $m$ and dimension $n$.
\begin{enumerate}
\item If $\sp(\B) \le \sp(\A)$, then $\Dg^{(j)}(\A) \subseteq \Dg^{(j)}(\B)$ for $j=0,\ldots,\ell-1$, and
$\Dg(\A)$ and $\Dg^{(0)}(\A)$ are respectively the subgroups of $\Dg(\B)$ and $\Dg^{(0)}(\B)$.

\item If $\sp(\A)=\sp(\B)$, then $\Dg(\A)=\Dg(\B)$, $\Dg^{(0)}(\A)=\Dg^{(0)}(\B)$.
If $\A,\B$ are further nonnegative and weakly irreducible, then $c(\A)=c(\B)$.
\end{enumerate}
\end{lem}

\begin{proof}
Suppose that $D=\hbox{diag}(d_{1}, \ldots, d_n) \in \Dg^{(j)}(\A)$ for a fixed $j$.
Then $\A=e^{-\i \frac{2\pi j}{\ell}}D^{-(m-1)}\A D$, or equivalently
\begin{equation}\label{comb1}
a_{i_1i_2\ldots i_m}=e^{-\i \frac{2\pi j}{\ell}} d_{i_1}^{-(m-1)} a_{i_1i_2\ldots i_m} d_{i_2} \cdots d_{i_m}, i_j \in [n], j \in [m].
\end{equation}
If $a_{i_1i_2\ldots i_m} \ne 0$, then
\begin{equation}\label{comb2}
e^{\i \frac{2\pi j}{\ell}} d_{i_1}^{m}=d_{i_1} d_{i_2} \cdots d_{i_m}.
\end{equation}
Since  $\sp(\B) \le \sp(\A)$, by (\ref{comb2}) we have
\begin{equation}\label{comb1}
b_{i_1i_2\ldots i_m}=e^{-\i \frac{2\pi j}{\ell}} d_{i_1}^{-(m-1)} b_{i_1i_2\ldots i_m} d_{i_2} \cdots d_{i_m}, i_j \in [n], j \in [m],
\end{equation}
which implies that $D \in \Dg^{(j)}(\B)$ and $\Dg^{(j)}(\A) \subseteq \Dg^{(j)}(\B)$.

By the above result, if $\sp(\A)=\sp(\B)$, obviously $\Dg(\A)=\Dg(\B)$ and  $\Dg^{(0)}(\A)=\Dg^{(0)}(\B)$.
If $\A,\B$ are further nonnegative and weakly irreducible, then
$\A$ is spectral $\ell$-symmetric if and only if $\B$ is  spectral $\ell$-symmetric by Theorem \ref{PF2}(2), implying that $c(\A)=c(B)$.
\end{proof}

\subsection{Incidence matrix of combinatorial symmetric tensors}

Let $\A$ be a combinatorial symmetric tensor of order $m$ and dimension $n$.
Set $$E(\A)=\{(i_1, i_2, \cdots, i_m)\in [n]^m: a_{i_1i_2\cdots i_m}\neq 0, 1\le i_1\le \cdots \le i_m \le n\}.$$
We define
\[b_{e,j}=|\{k: i_k=j, e=(i_1, i_2, \cdots, i_m) \in E(\A), k \in [m]\}|\]
and obtain an $|E(\A)|\times n$ matrix $B_\A=(b_{e,j})$, which is called the \emph{incidence matrix} of $\A$.
Observe that $b_{ij}\in \{0, 1, \cdots, m-1, m\}$, and the incidence matrix $B_\A$
can be viewed as one over $\Z_m$, where $\Z_m$ is the ring of integers modulo $m$.

\subsection{Composition length of modules}

Let $R$ be a ring with identity and $M$ a nonzero module over $R$. Recall that a finite chain of $l+1$ submodules of $M$
\[M=M_0>M_1>\cdots>M_{l-1}>M_l=0\]
is called a \emph{composition series} of length $l$ for $M$ provided that $M_{i-1}/M_i$ is simple for each $i \in [l]$.
By the Jordan-H\"older Theorem, if a module $M$ has a finite composition series,
then every pair of composition series for $M$ are equivalent. To be precise, if
\[M=N_0>N_1>\cdots>N_{k-1}>N_k=0\]
is another composition series for $M$, then $k=l$ and there is a permutation $\sigma \in S_l$ such that
\[M_i/M_{i+1}\cong N_{\sigma(i)}/N_{\sigma(i)+1}, i \in [l].\]
The length of the composition series for such a module $M$ is called the \emph{composition length} of $M$,
denoted by $\cl(M)$; see \cite[Section 11]{AF} for more details.

Next we focus on the modules over $\Z_m$. Let $M$ be a finitely generated module over $\Z_m$.
Clearly, $M$ is a $\Z$-module (i.e. an abelian group) and $mx=0$ for any $x\in M$.
Suppose that $M$ is isomorphic to
\[\Z_{p_1^{k_1}}\oplus \Z_{p_2^{k_2}}\oplus \cdots \oplus \Z_{p_s^{k_s}},\]
as a $\Z$-module, where $p_i$ is a prime for $i \in [s]$.
Since $mx=0$ for any $x\in M$, we have $p_i^{k_i}\mid m$ for all $i \in [s]$.
In this situation, $\cl(M)=\sum_{i=1}^s k_i$.
In particular, if $m$ is prime, then $\Z_m$ is a field, and $M$ is a linear space over $\Z_m$ whose dimension is exactly the composition length of $M$.
If $m=p_1^{k_1}p_2^{k_2}\cdots p_s^{k_s}$ with $p_i$'s being prime and distinct,
then as a regular module over $\Z_m$, $\cl(\Z_m)=\sum_{i=1}^s k_i=:\cl(m)$.
Also, if $d|m$, then $\Z_d$ is a submodule of the above regular module, and $\cl(\Z_d)=:\cl(d)$.

\subsection{Hypergraphs}

A \emph{hypergraph} $G=(V(G),E(G))$ consists of a vertex set  $V(G)=\{v_1, v_2, \ldots, v_n\}$
and an edge set $E(G)=\{e_1, e_2, \ldots, e_{l}\}$ where $e_{j}\subseteq V(G)$ for $j \in [l]$.
If $|e_{j}|=m$ for each $j\in [l]$, then $G$ is called an {\it $m$-uniform} hypergraph.
In particular, the $2$-uniform hypergraphs are exactly the classical simple graphs.

Now suppose $G$ is $m$-uniform.
The {\it adjacency tensor} $\A(G)$ of the hypergraph $G$ is defined as an $m$-th order $n$-dimensional tensor $\mathcal{A}(G)=(a_{i_{1}i_{2}\ldots i_{m}})$,
where
\[a_{i_{1}i_{2}\ldots i_{m}}=
\begin{cases}
\frac{1}{(m-1)!}, &  \mbox{if~} \{v_{i_{1}},v_{i_{2}},\ldots,v_{i_{m}}\} \in E(G);\\
  0, & \mbox{otherwise}.
\end{cases}
\]
Observe that the adjacency tensor $\A(G)$ of a hypergraph $G$ is nonnegative and symmetric, and it is weakly irreducible if and only if $G$ is connected \cite{PF, YY3}.
The \emph{incidence matrix} of $G$, denoted by $B_G=(b_{e,v})$, coincides with that of $\A(G)$, that is
\begin{align*}
b_{e,v}=\begin{cases}
1, & \mbox{if~} \ v\in e,\\
0, & \textrm{otherwise}.
\end{cases}
\end{align*}
For more knowledge on hypergraphs and spectral hypergraph theory, one can refer to \cite{Ber} and \cite{QiLuo}.

\section{Eigenvariety associated with eigenvalues with modulus equal to spectral radius}
Let $\A$ be a nonnegative weakly irreducible tensor, which is spectral $\ell$-symmetric.
Let $\V_j=\V_j(\A)$ (respectively, $\PV_j=\PV_j(\A)$) be the eigenvariety (respectively, projective eigenvariety) of $\A$ associated with the eigenvalue $\la_j=\rho(\A)e^{\i \frac{2\pi j}{\ell}}$ for $j=0,1,\ldots,\ell-1$.
Set $\V=\cup_{j=0}^{\ell-1} \V_j$ and $\PV=\cup_{j=0}^{\ell-1} \PV_j$.

For each $\y \in \V$, by Lemma \ref{ev}, $y_i\neq 0$ for all $i\in [n]$.
 Set $\bar\y=y_1^{-1}\y \in \PV$.
Therefore, we can assume each point $\y \in \PV$ holds $y_1=1$.
Also by Lemma \ref{ev}, for each $\y \in \PV_j$, $|\y|=\v_p$ and $\y=D_{\y}\v_p$, where $\v_p$ is the unique Perron vector in $\PV_0$, and $D_\y$ is defined as in (\ref{Dy});
and by Theorem \ref{PF2} (2),
\begin{equation}\label{dig-sim}\A=e^{-\i \frac{2 \pi j}{\ell}} D_\y^{-(m-1)}\A D_\y.\end{equation}

We define a quasi-Hadamard product $\circ$ in $\PV$ as follows:
 \begin{equation}\label{product} \y \circ \hat{\y}:=D_\y D_{\hat{\y}} \v_p.\end{equation}

\begin{lem} \label{VtoD}
The $(\PV,\circ)$ is an abelian group which contains $\PV_0$ as a subgroup and $\PV_j$ as a coset of $\PV_0$ for each $j \in [\ell-1]$,
and there exists a group isomorphism $\Psi: \PV \to \Dg$ which sends $\PV_j$ to $\Dg^{(j)}$ for $j=0,1,\ldots,\ell-1$.
\end{lem}

\begin{proof}
Suppose $\y \in \PV_i$, $\hat{\y} \in \PV_j$.
Then by (\ref{dig-sim}), $$\A= e^{-\i \frac{2 \pi (i+j)}{\ell}} (D_\y D_{\hat{\y}})^{-(m-1)}\A (D_\y D_{\hat{\y}}).$$
So $ D_\y D_{\hat{\y}} \in \Dg^{(i+j)}$, and hence $\y \circ \hat{\y} \in \PV_{i+j}$ by (\ref{2nd}) and (\ref{product}), where $i+j$ is taken modulo $m$.
In addition, by Lemma \ref{group}(2), $D_\y^{-1} \in \Dg^{(-i)}$, hence $\y^{\circ (-1)}=D_\y^{-1}\v_p \in \PV_{-i}$, where $(-i)$ is also taken modulo $m$.
It is easily seen that $\v_p$ is the zero in $\PV$.
So $(\PV,\circ)$ is an abelian group. From the above discussion one can see that $\PV_0$ is a subgroup of $\PV$.

Define
\begin{equation}\label{psi}
\Psi: \PV \to \Dg \;(\y \mapsto D_\y).
\end{equation}
By Lemma \ref{group}(2) and (\ref{2nd}), $\Psi$ is a bijection.
It is easy to see that \[\Psi(\y \circ \hat{\y})=\Psi(\y) \Psi(\hat{\y}).\]
So $\Psi$ is a group isomorphism and sends $\PV_j$ to $\Dg^{(j)}$ by (\ref{2nd}),
  which implies that $\PV_j$ is a coset of $\PV_0$
  as $\Dg^{(j)}$ is a coset of $\Dg^{(0)}$ by Lemma \ref{group} for each $j \in [\ell-1]$.
\end{proof}

Now we further assume $\A$ is also combinatorial symmetric.
So, for each $\D \in \Dg$, $D^m=\I$ by Lemma \ref{group}(3).
It follows that there exists a unique $\phi_i\in \{0, 1, \cdots, m-1\}$ for each $i\in [n]$
   such that $D=\diag\{e^{\i \frac{2\pi}{m} \phi_1}, \ldots, e^{\i \frac{2\pi}{m} \phi_n}\}$.
Set
\begin{equation}\label{phiD}
\phi_D=(\phi_1,\ldots,\phi_n) \in \mathbb{Z}_m^n.
\end{equation}
Therefore, we have an injective map:
\begin{equation} \label{Phi}
\Phi: \Dg \to \mathbb{Z}_m^n \;(D \mapsto \phi_D).
\end{equation}
Furthermore, $\Phi(D\hat{D})=\Phi(D)+\Phi(\hat{D})$.
Since $\y^{\circ m}:=(D_\y)^m\v_p=\v_p$, $\PV$, $\Dg$, and $\mathbb{Z}_m^n$ all admit $\Z_m$-modules,
and $\Psi$ is a $\Z_m$-isomorphism from $(\PV,\circ)$ to $(\Dg, \cdot)$, and $\Phi$ is a $\Z_m$-monomorphism from $(\Dg, \cdot)$ to $(\Z_m^n, +)$.

In the following, we will determine the image $\Phi(\Dg)$.
Denote $\mathbbm{1}=(1, 1,\ldots, 1)\in \Z_m^n$ the all-one vector.

\begin{lem} \label{VinS}
Let $\A$ be a nonnegative combinatorial symmetric weakly irreducible tensor, which is spectral $\ell$-symmetric.
If $\y \in \V_j$, then $B_\A \Phi\Psi(\y)=\frac{m j}{\ell}\mathbbm{1}$ over $\Z_m$ for each $j=0, 1, \cdots, \ell-1$, where $B_\A$ is the incidence matrix of $\A$.
\end{lem}

\begin{proof}
Let $\y=(y_1, \cdots, y_n)$ be an eigenvector in $\V_j$. Then by Theorem \ref{PF2}, we have
\[\A=e^{-\i\frac{2\pi}{\ell}j}D_\y^{-(m-1)}\A D_\y\]
where $D_\y=\diag(e^{\i \frac{2\pi}{m} \phi_1},\cdots, e^{\i \frac{2\pi}{m} \phi_n})$
   and $\frac{y_i}{|y_i|}=e^{\i \frac{2\pi}{m} \phi_i}$ for $i \in [n]$.
To be precise, for any $i_1, \cdots, i_m\in [n]$,
\[
a_{i_1i_2\cdots i_m}=e^{-\i\frac{2\pi}{\ell}j}e^{\i\frac{2\pi}{m}(-(m-1)\phi_{i_1})}a_{i_1i_2\cdots i_m}e^{\i\frac{2\pi}{m}(\phi_{i_2}+\cdots+\phi_{i_m})}
\]
Therefore, if $a_{i_1i_2\cdots i_m}\neq 0$, then
\begin{align*}
\phi_{i_1}+\cdots+\phi_{i_m}\equiv \frac{m}{\ell}j \mod m.
\end{align*}
It follows that $\sum_{k=1}^n b_{e,k}\phi_k \equiv \frac{m}{\ell}j \mod m$
  and $B_\A\Phi\Psi(\y)=\frac{m j}{\ell}\mathbbm{1}$ over $\Z_m$.
\end{proof}

Denote by
\begin{equation}\label{equation} \S_j=\S_j(\A)=\{\x \in \Z_m^n: B_\A\x=\frac{mj}{\ell}\mathbbm{1} \hbox{~over~} \Z_m\}, j=0,1,\ldots,\ell-1.\end{equation}
Observe that $\Phi\Psi(\V_j) \subseteq \S_j$ by Lemma \ref{VinS} for $j=0,1,\ldots,\ell-1$.
Furthermore, $\S_0$ is a $\Z_m$-submodule of $\Z_m^n$,
  and $\S_j=\y^{(j)}+\S_0=\{\y^{(j)}+\y: \y\in \S_0\}$ for $j \in [\ell-1]$, where $\y^{(j)} \in \S_j$.

Observe that $\mathbbm{1} \in \S_0$ since the sum of each row of $B_\A$ is equal to $m$.
If $\phi \in \S_j$, so is $\phi+\alpha \mathbbm{1}$ for any $\alpha \in \Z_m$.
So we only consider
\[\PS_j=\PS_j(\A)=\{\phi \in \S_j: \phi_1=0\}, j=0,1,\ldots,\ell-1.\]
Observe that $\Phi\Psi(\PV_j) \subseteq \PS_j$ as $y_1=1$ implies that $(\Phi\Psi(\y))_1=0$ for any $\y \in \PV_j$   by the definitions (\ref{Dy}) and (\ref{phiD}).
Now $\PS_0$ is a $\Z_m$-submodule of $\S_0$;
   in fact, it is isomorphic to the quotient module $\S_0/(\Z_m\mathbbm{1})$.
We also have $\PS_j=\y^{(j)}+\PS_0$ for $j \in [\ell-1]$, where $\y^{(j)} \in \PS_j$.

In the following, we will show that $\Phi(\Dg^{(j)})=\PS_j$ for each $j=0,1,\ldots,\ell-1$,
and hence $\Phi(\Dg)=\PS:=\cup_{j=1}^{\ell-1}\PS_j$.

\begin{lem} \label{DtoS}
Let $\A$ be a nonnegative combinatorial symmetric weakly irreducible tensor which is spectral $\ell$-symmetric.
Then $\Phi$ is a $\Z_m$-isomorphism from $\Dg$ to $\PS$ which sends $\Dg^{(j)}$ to $\PS_j$ for each $j=0,1,\ldots,\ell-1$.
\end{lem}

\begin{proof}
By Lemma \ref{VinS}, $\Phi\Psi(\PV_j) =\Phi(\Dg^{(j)}) \subseteq \PS_j$.
For each $\phi=(\phi_1, \phi_2, \cdots, \phi_n)\in \PS_j$, as $B_\A \phi=\frac{mj}{\ell}\mathbbm{1}$,
if $a_{i_1 i_2 \cdots i_m} \ne 0$, then
\begin{equation}\label{the sum of phi}
\phi_{i_1}+\cdots+\phi_{i_m}\equiv \frac{mj}{\ell} \mod m.
\end{equation}
Let $D_\phi=\diag(e^{\i \frac{2\pi}{m} \phi_1}, \cdots, e^{\i \frac{2\pi}{m} \phi_n})$.
  Then
  $$\A=e^{-\i \frac{2 \pi j}{\ell}} D_\phi^{-(m-1)} \A D_\phi.$$
As $\phi_1=0$,  $(D_\phi)_{11}=1$, implying $D_\phi \in \Dg^{(j)}$.
\end{proof}

By Lemmas \ref{VtoD} and \ref{DtoS}, we get the main result of this paper.

\begin{thm}\label{VtoS}
Let $\A$ be a nonnegative combinatorial symmetric weakly irreducible tensor which is spectral $\ell$-symmetric.
Then $\PV_0(\A)$ is $\Z_m$-module isomorphic to $\PS_0(\A)$.
In particular, if $m$ is a prime, then $\PV_0(\A)$ is a linear space over $\Z_m$,
 and $\PV_j(\A)$ is an affine space over $\Z_m$ for $j \in [\ell-1]$.
\end{thm}

Theorem \ref{VtoS} provides an algebraic structure on the projective eigenvariety $\PV_0(\A)$
  associated with $\rho(\A)$, which can be obtained explicitly by solving the equation $B_\A \x=0$ over $\Z_m$.
As a byproduct, since $\Dg^{(0)}$ is $\Z_m$-isomorphic to $\PS_0$,
 the structure of $\Dg^{(0)}$ can also be obtained explicitly, which cannot be discussed in \cite{FHB}.

\begin{Def}\label{sg}
Let $\A$ be a nonnegative combinatorial symmetric weakly irreducible tensor of order $m$.
Then the composition length of the $\Z_m$-module $\PV_{\rho(\A)}$  is called
the \emph{stabilizing dimension} of $\A$, denoted by $\gamma(\A)$.

In particular, if $m$ is a prime, then $\Z_m$ is a field, and $\gamma(\A)$ is exactly the dimension of the $\Z_m$-linear space $\PV_{\rho(\A)}$.
\end{Def}

We now determine the structure of $\PS_0$.
Similar to the Smith normal form for a matrix over a principle ideal domain, one can also deduce a diagonal form for a matrix over $\Z_m$ by some elementary row transformations and column transformations, that is, for any matrix $B \in \Z_m^{k \times n}$,
there exist two invertible matrices $P \in \Z_m^{k \times k}$ and $Q \in \Z_m^{n \times n}$ such that
\begin{equation} \label{smith}
PBQ=\begin{pmatrix}
d_1 & 0 & 0 &  & \cdots & & 0\\
0 & d_2 & 0 &  & \cdots & &0\\
0 & 0 & \ddots &  &  & & 0\\
\vdots &  &  & d_r &  & & \vdots\\
 & & & & 0 & & \\
  & & & &  & \ddots & \\
0 &  &  & \cdots &  & &0
\end{pmatrix},
\end{equation}
where $r \ge 0$, $ 1 \le d_i \le m-1$, $d_i | d_{i+1}$ for $i=1,\ldots, r-1$, and $d_i |m$ for all $i=1,2,\ldots,r$.
The matrix in (\ref{smith}) is called the {\it Smith normal form} of $B$ over $\Z_m$,
$d_i$'s are the \emph{invariant divisors} of $B$ among of which the invariant divisors $1$ are called the \emph{unit invariant divisors} of $B$.

\begin{thm} \label{stru}
Let $\A$ be a nonnegative combinatorial symmetric weakly irreducible tensor of order $m$ and dimension $n$.
Suppose that $B_\A$ has a Smith normal form over $\Z_m$ as in (\ref{smith}).
Then $1 \le r \le n-1$, and
\[\PS_0(\A) \cong \oplus_{i, d_i \ne 1} \Z_{d_i} \oplus \underbrace{\Z_m \oplus \cdots \oplus \Z_m}_{n-1-r \text{\em \small~copies~}}\]
Consequently,
\[ s(\A)=m^{n-1-r} \Pi_{i=1}^r d_i, \;  \gamma(\A)=\sum_{i, d_i \ne 1} \cl(d_i)+ (n-1-r)\cl(m),\]
and $s(\A)|m^{n-1}$, $s(\A)< m^{n-1}$, $\gamma(\A) < (n-1)\cl(m)$.
\end{thm}

\begin{proof}
As $\A$ is irreducible, $B_\A \ne 0$, implying that $r \ge 1$.
Also, as $B_\A \mathbbm{1} =0$, adding all other columns to the last column of $B_\A$, the last column of the resulting matrix becomes a zero column,
   which keeps invariant under any elementary row transformations and column transformations.
So $PB_\A Q$ always has a zero column, implying that $r \le n-1$.

Let $\S'_0=\{\x: P B_\A Q \x=0 \hbox{~over~} \Z_m\}$.
Then $\x \in \S_0=\S_0(\A)$ if and only if $Q^{-1} \x \in \S'_0$, implying that $\S_0$ is $\Z_m$-module isomorphic to $\S'_0$.
It is easy to see
$$\S'_0 \cong \oplus_{i, d_i \ne 1} \Z_{d_i} \oplus \underbrace{\Z_m \oplus \cdots \oplus \Z_m}_{n-r \text{\small~copies~}}.$$
So $$\PS_0(\A) \cong \S_0/(\Z_m \mathbbm{1}) \cong \S'_0/(\Z_m (Q^{-1}\mathbbm{1})) \cong
\oplus_{i, d_i \ne 1} \Z_{d_i} \oplus \underbrace{\Z_m \oplus \cdots \oplus \Z_m}_{n-1-r \text{\small~copies~}}.$$
Since $\PV_0(\A)$ is $\Z_m$-module isomorphic to $\PS_0(\A)$ by Theorem \ref{VtoS},
 we get the remaining results by Lemma \ref{group}(2) and Definition \ref{sg}.
\end{proof}

\begin{cor} \label{stru-unit}
Let $\A$ be a nonnegative combinatorial symmetric weakly irreducible tensor of order $m$ and dimension $n$.
Suppose that $B_\A$ has $t$ unit invariant divisors.
Then $$s(\A)\le m^{n-1-t},\gamma(\A) \le (n-1-t)\cl(m),$$
both with equalities if and only if $B_\A$ has exactly $t$ invariant divisors all being $1$.

In particular, if $m$ is prime, then
$$s(\A)= m^{n-1-\rank B_\A},\gamma(\A) = n-1-\rank B_\A,$$
where $\rank B_\A$ is the rank of $B_\A$ over the field $\Z_m$.
\end{cor}

\begin{proof}
The first part of the result follows from Theorem \ref{stru} immediately.
For the second part, as $m$ is prime, $\cl(m)=1$, the invariant divisors of $B_\A$ are all unit invariant divisors $1$, the number of which is exactly the rank of $B_\A$.
Note that $\PS_0(\A)$ is a $\Z_m$-linear space of dimension $\gamma(\A)$ in this case.
The result follows.
\end{proof}

If $m$ is prime, the bounds in Corollary \ref{stru-unit} hold as equalities.
The tightness of the bounds in Corollary \ref{stru-unit} for the general case will be further discussed in Section 4.

\begin{cor}\label{stru-1}
Let $\A$ be a nonnegative combinatorial symmetric weakly irreducible tensor of order $m$ and dimension $n$.
The the following conditions are equivalent.

\begin{enumerate}
\item $s(\A)=1$, i.e. $\A$ has only one eigenvector corresponding to $\rho(\A)$.

\item $\gamma(\A)=0$.

\item $\PS_0(\A)=\{0\}$.

\item $\S_0(\A)=\{\alpha \mathbbm{1}: \alpha \in \Z_m\}$.

\item $B_\A$ has exactly $n-1$ invariant divisors all being $1$.
\end{enumerate}
\end{cor}

\begin{rmk}
Let $\A$ be a nonnegative combinatorial symmetric weakly irreducible tensor of order $m$ and $n$ dimension.
Then $s(\A)$ is independent of the diagonal entries of $\A$ as each nonzero diagonal entry make no contribution to $\PS_0(\A)$ by Theorem \ref{stru}.
This fact can also be seen from the definition of $\Dg^{(0)}(\A)$.
In particular $s(\A)=s(\I +\A)$.
\end{rmk}

\begin{exm}
Let $\A$ be a positive tensor of order $m$ and dimension $n$.
Then $s(\A)=1$.
Let $\x \in \Z_m^n$ be a solution of $B_\A \x=0$.
As $\A$ contains nonzero entries $a_{i,i+1,\ldots,i+1}$ for $i=1,\ldots, n-1$,
we have $x_i+(m-1)x_{i+1}=0$, and hence $x_1=\cdots=x_n$.
So the assertion follows by Corollary \ref{stru-1}.
\end{exm}

\begin{exm}\label{ex1}
Let $\A$ be a $12$th order $6$-dimensional combinatorial symmetric tensor with the incidence matrix
\[B_\A=\begin{pmatrix}
3 & 3 &3 & 1 & 1 & 1\\
1 & 3 &3 & 3 & 1 & 1\\
1 & 1 &3 & 3 & 3 & 1
\end{pmatrix}.\]
Then
\[\PS_0(\A)=\Z_m {\y_1}\oplus \Z_m {\y_2}\oplus \Z_m {\y_3}
\oplus \Z_m {\y_4},\]
where
\begin{align*}
\y_1=(0, 0, 6, 6, 6, 6)^T, &\  \y_2=(0, 6, 6, 6, 6, 0)^T, \\
\y_3=(0, 0, 1, 0, 0, 9)^T, &\  \y_4=(0, 1, 0, 0, 1, 8)^T.
\end{align*}
As a $\Z_{12}$-module, we have
$\PS_0(\A) \cong \Z_2\oplus \Z_2\oplus \Z_{12}\oplus \Z_{12}$, implying that $s(\A)=576$ and $\gamma(\A)=8$.

In fact, the Smith normal form of $B_\A$ is
\[\begin{pmatrix}
1 & 0 &0& 0 & 0 & 0\\
0 & 2 &0 & 0 & 0 & 0\\
0 & 0 &2 & 0 & 0 & 0
\end{pmatrix}.\]
So we easily get the above result.
\end{exm}

%
%

\begin{exm}\label{ex2}
Let $\A$ be a 3rd order $6$-dimensional tensor with $a_{i_1i_2i_3}=1$ for
   $\{i_1, i_2, i_3\}\in \{\{1, 2, 3\},\{2, 3, 4\}, \{3, 4, 5\}, \{4, 5, 6\},\{5, 6, 1\}, \{6, 1, 2\}\}$
  and $a_{i_1i_2i_3}=0$ otherwise.
The incidence matrix $B_\A$ has rank $4$ over $\Z_3$, so $\gamma(\A)=1$ and $s(\A)=3$ by Corollary \ref{stru-unit}.
$\PS_0(\A)$ has a $\Z_3$-basis $\y=(0, 1, 2, 0, 1, 2)^T$.
\end{exm}

By Lemma \ref{comb}, the group $\Dg^{(0)}(\A)$ is only determined by the support of $\A$.
Finally we investigate the stabilizing index and the stabilizing dimension of a subpattern of $\A$.


\begin{thm} \label{pattern}
Let $\A,\hat\A$ be nonnegative combinatorial symmetric weakly irreducible tensors of order $m$ and dimension $n$
  such that $\sp(\hat\A) \le \sp(\A)$.
Then $\PS_0(\A)$ is a $\Z_m$-submodule of $\PS_0(\hat\A)$.
Consequently, $s(\A) | s(\hat\A)$ and $\gamma(\A) \le \gamma(\hat\A)$.
\end{thm}

\begin{proof}
As $\sp(\hat\A) \le \sp(\A)$, $B_{\hat\A}$ is a submatrix of $B_\A$ with rows indexed by the nonzero entries of $\hat\A$ and columns same as $B_\A$.
So, if $\x \in \PS_0(\A)$, surely $\x \in \PS_0(\hat\A)$, which implies that
$\PS_0(\A)$ is a $\Z_m$-submodule of $\PS_0(\hat\A)$.
The result follows by Theorem \ref{stru} immediately.
\end{proof}

\section{Application to hypergraphs}

Let $G=(V, E)$ be a connected $m$-uniform hypergraph and $\A(G)$ be the adjacency tensor of $G$.
Clearly, $\A(G)$ is a nonnegative symmetric weakly irreducible tensor.
The stabilizing index and the stabilizing dimension of $G$ are referred to those of
  $\A(G)$, denoted by $s(G)$ and $\gamma(G)$, respectively.

A {\it cored hypergraph} \cite{HQS} is one such that each edge contains a cored vertex, i.e. the vertex of degree one.
Let $m \ge 3$ be an integer, and let $G$ be a simple graph.
The {\it  $m$-th power} of $G$ \cite{HQS}, denoted by $G^m$,
is obtained from $G$ by replacing each edge (a $2$-set) with a $m$-set by adding $m-2$ additional vertices.
Obviously, $G^m$ is a cored hypergraph.

\begin{thm}\label{cored}
Let $H$ be a connected $m$-uniform cored hypergraph on $n$ vertices with $t$ edges.
Then $s(H)= m^{n-1-t}$, $\gamma(H) = (n-1-t)\cl(m)$.
\end{thm}

\begin{proof}
Let $e_1,\ldots,e_t$ be the $t$ edges of $H$.
Then $H$ contains at least $t$ cored vertices, say $1,\ldots,t$, where $i \in e_i$ for $i \in [t]$.
So, the incidence matrix $B_H$ has an $t \times t$ identity submatrix with columns indexed by $1,\ldots,t$,
which implies that $B_H$ has exactly $t$ invariant divisors all being $1$.
The result follows by Corollary \ref{stru-unit}.
\end{proof}

\begin{cor}\label{power}
Let $G$ be a connected simple graph on $n$ vertices with $t$ edges,
and let $G^m$ be the $m$-th power of $G$.
Then $s(G^m)= m^{n-1+t(m-3)}$, $\gamma(G^m) = [n-1+t(m-3)]\cl(m)$.
\end{cor}

\begin{proof}
By the definition, $G^m$ has $n+t(m-2)$ vertices and $t$ edges.
The result follows by Theorem \ref{cored}.
\end{proof}

\begin{exm}\label{linepath}
Let $P_n$ be a simple path on $n$ vertices (as simple graphs).
The $m$-th power $P_n^m$ of $P_n$ is called a \emph{hyperpath} \cite{HQS}.
It is easily seen that $s(P_n^m)= m^{(n-1)(m-2)}$, $\gamma(P_n^m) = (n-1)(m-2)\cl(m)$ by taking $t=n-1$ in Corollary \ref{power}.
%
%
Obviously, $P_2^m$ consists of exactly one edge, $s(P_2^m)=m^{m-2}$.
\end{exm}

\begin{exm}
An $m$-uniform \emph{complete hypergraph} on $n$ vertices, denoted by $K_n^{[m]}$, is a hypergraph with any $m$-subset of the vertex set being an edge.
Suppose that $n \ge m+1$.
 For any two vertices $i \ne j$, taking an arbitrary $(m-1)$-set $U$ of $K_n^{[m]}$ which does not contain $i$ or $j$,
by considering the equation $B_{K_n^{[m]}}\x =0$ on two edges $U \cup \{i\}$ and $U \cup \{j\}$,
we have $x_i=x_j$.
So the equation only has the solutions $\alpha \mathbbm{1}$ for some $\alpha \in \Z_m$, implying $s(K_n^{[m]})=1$ by Corollary \ref{stru-1}.
In particular, if $n=m+1$, then $K_{m+1}^{[m]}$ is called an \emph{$m$-simplex} \cite{CD}, and surely $s(K_{m+1}^{[m]})=1$.
\end{exm}

\begin{exm}\label{gsquid}
An $m$-uniform \emph{generalized squid}, denoted by $S(m,t)$, where $1 \le t \le m$, is obtained from an edge by attaching $t$ edges
to $t$ vertices in the edge respectively; see Fig. \ref{squid}.
In particular, $S(m,m-1)$ is called a \emph{squid} \cite{HQS}.
By a direct computation, the incidence matrix $B_{S(m,t)}$ has $t+1$ invariant divisors all being $1$.
So, $s(S(m,t))=m^{(t+1)(m-2)}$, $\gamma(S(m,t))=(t+1)(m-2)\cl(m)$.
\end{exm}

\begin{figure}[htbp]
\includegraphics[scale=0.66]{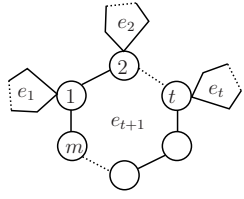}
\caption{\small An $m$-uniform generalized squid}\label{squid}
\end{figure}

Let $H=(\hat{V},\hat{E})$ be a connected spanning sub-hypergraph of $G$, that is, $\hat{V}=V$ and $\hat{E} \subseteq E$.
Then $\A(H)$ is also a nonnegative symmetric weakly irreducible tensor with same order and dimension as $\A(G)$,
   and  $\sp(\A(H)) \le \sp(\A(G))$.
So we get the following result by Theorem \ref{pattern}.

\begin{thm}\label{span-sub}
Let $G$ be a connected uniform hypergraph and let $H$ be a connected spanning sub-hypergraph of $G$.
Then $s(G) | s(H)$ and $\gamma(G) \le \gamma(H)$.
\end{thm}


\subsection{3-uniform hypergraphs}
In this subsection, we mainly deal with connected $3$-uniform hypergraphs $G$.
As $3$ is prime, $s(G)=3^{\gamma(G)}$.
So we only discuss the parameter $\gamma(G)$.

For convenience to depicting or drawing a $3$-uniform hypergraph $G$, we will use a triangle to represent an edge of $G$,
  where the vertices of triangles (drawing as small circles) are same as the vertices of $G$.
Two triangles sharing one common edge (respectively, one vertex) means that the corresponding two edges sharing two vertices of the common edge(respectively, the vertex).

\begin{prop} \label{3unif}
Let $G=(V, E)$ be a $3$-uniform connected hypergraph. Then $\gamma(G)>0$ if and only if there exists a nontrivial tripartition
$V=V_0\cup V_1 \cup V_2$, such that for each $e\in E$, $e\subseteq V_i$ for some $i \in \{0,1,2\}$, or $e\cap V_j\neq \emptyset$ for each $j=0, 1, 2$; see Fig. \ref{stru-3}.
\end{prop}

\begin{proof} Let $[n]$ be the set of vertices of $G$.
Suppose $\gamma(G)>0$. Then the equation $B_{G}\x=0$ over $\Z_3$ has a solution
$\phi=(\phi_1, \phi_2, \cdots, \phi_n) \in \Z_3^n$ such that $\phi_1=0$ and $\phi_j\neq 0$ for some $j\in [n]$ by Corollary \ref{stru-1}.
Denote $V_i=\{j\in V: \phi_j=i\}$, $i=0, 1, 2$.
If $e=\{r, s, t\}\in E$, then $\phi_{r}+\phi_{s}+\phi_{t}=0$ over $\Z_3$.
It follows that $e \subseteq  V_i$ for some $i \in \{0,1,2\}$, or $e \cap V_j \neq \emptyset$ for each $j=0, 1, 2$.
Each $V_i$ is nonempty; otherwise $G$ is not connected.
So $V_0,V_1,V_2$ form a nontrivial tripartition of $V$.

Conversely, assume that there exits a nontrivial tripartition $V=V_0\cup V_1 \cup V_2$ satisfying the condition of the proposition.
Set $\phi=(\phi_1, \cdots, \phi_n)$, where $\phi|_{V_i}=i$ for $i=0, 1, 2$.
It is easily seen that $\phi$ is a solution of $B_{\A(G)}\x=0$ over $\Z_3$.
Clearly, $\phi\notin \Z_3\mathbbm{1}$.
So $\gamma(G)>0$.
\end{proof}

\begin{figure}[htbp]\label{stru-3}
\includegraphics[scale=0.6]{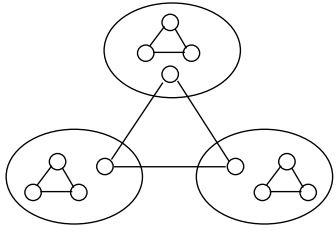}
\caption{\small Structure of $3$-uniform connected hypergraphs with positive stabilizing dimension}\label{stru-3}
\end{figure}

A tripartition of a $3$-uniform connected hypergraph $G$ satisfying Proposition \ref{3unif}
is equivalent to a surjective \emph{labeling}(or a map)
\begin{equation}\label{label}
f: V(G) \to \{0,1,2\}
\end{equation}
 such that each edge of $G$ has three vertices with the same label or pairwise different labels, and there exists at least one edge
 of $G$ which has three vertices with  pairwise different labels.

\begin{exm}
Let $G$ be an \emph{$n$-wheel}($n \ge 3$), a $3$-uniform hypergraph with $n$ edges illustrated in Fig. \ref{wheel}.
By labeling the vertices of $G$, it is easy to verified that $\gamma(G)>0$ if and only if $n$ is even.
An $n$-wheel is said {\it odd} (respectively, {\it even}) if $n$ is  {\it odd} (respectively, {\it even}).
By Theorem \ref{span-sub}, the $3$-simplex has zero stabilizing dimension as it contains a $3$-wheel as a spanning sub-hypergraph.
\end{exm}

\begin{figure}[htbp]
\includegraphics[scale=0.6]{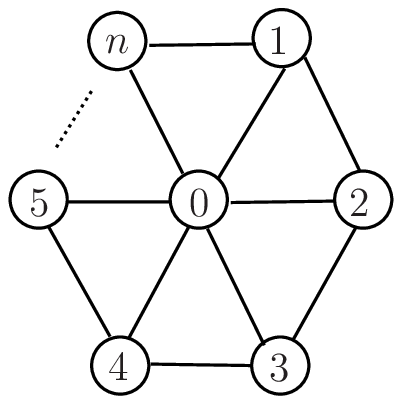}
\caption{\small $n$-wheel}\label{wheel}
\end{figure}

\begin{exm}
Let $G$ be a connected $3$-uniform hypergraph with a pendant triangle (edge)
(i.e. the triangle sharing only one vertex with other triangles). Then $\gamma(G)>0$.
Let $\{v_0,v_1,v_2\}$ be the pendant triangle of $G$, where $v_3$ also belongs to the other triangles $G$.
Labeling $v_0$ with $0$, $v_1$ with $1$, and all other vertices with $2$. Such labeling holds (\ref{label}).
\end{exm}

\begin{exm} \label{infinity}
We show that there exists a sequence of infinite many connected $3$-uniform hypergraphs $G$ with $\gamma(G)>0$ (respectively, $\gamma(G)=0$).
Let $\{G_i\}$ be a sequence of $3$-uniform connected hypergraphs constructed as follows.
\begin{itemize}
\item[(1)] $G_0$ is a connected $3$-uniform hypergraph.

\item[(2)] Suppose $G_i$ is constructed for $i \ge 0$. The hypergraph $G_{i+1}$ is obtained from $G_i$ by adding a new triangle which shares exactly one edge with the triangles of $G_i$.
\end{itemize}

\noindent Then $\gamma(G_i) >0$ (respectively, $\gamma(G_i) =0$) for all $i \ge 1$ if and only if $\gamma(G_0)>0$ (respectively, $\gamma(G_0)=0$).

Suppose $\gamma(G_0)>0$.
Then $G_0$ has a labeling $f$ holding (\ref{label}).
We will extend this labeling to $G_1$.
Let $\triangle=\{u,v,w\}$ be the triangle of $G_1$ added to $G_0$,
   where the edge $uv$  of $\triangle$ is also shared by some triangles of $G_0$.
If $f(u)=f(v)$, set $f(w)=f(u)$; otherwise, set $f(w)$ to be the only element in $\{0,1,2\}\backslash \{f(u),f(v)\}$.
So we get a labeling of $G_1$ holding (\ref{label}). Similarly, by induction we can show $\gamma(G_i) >0$ for all $i \ge 1$.

Next suppose $\gamma(G_0) =0$.
Assume that $\gamma(G_i) >0$ for some $i \ge 1$.
Then $G_i$ has a labeling $f$ holding (\ref{label}).
Let $\triangle=\{u,v,w\}$ be the triangle of $G_i$ added to $G_{i-1}$,
  where the edge $uv$  of $\triangle$ is  shared by some triangles of $G_{i-1}$.
If $f(u)=f(v)=f(w)$, then $f|_{G_{i-1}}$ is also a labeling holding (\ref{label}).
Otherwise, there exists a triangle $\hat\triangle=\{u,v,\hat{w}\}$ of $G_{i-1}$ whose vertices receive pairwise different labels ($f(\hat{w})=f(w)$), which implies that
$f|_{G_{i-1}}$ is also a labeling holding (\ref{label}).
So, by induction, we finally get that $G_0$ has a labeling holding (\ref{label}), a contradiction.

The assertion now follows by taking $G_0$ to be an even wheel or odd wheel.
\end{exm}

\subsection{Path cover}

Let $G=(V, E)$ be an $m$-uniform hypergraph.
Recall a \emph{walk} $W$ in $G$ of length $l$ is a nonempty alternating sequence
\[W: v_1, e_1, v_2, e_2, \cdots, v_l, e_l, v_{l+1}\]
of vertices and edges in $G$ such that $\{v_i, v_{i+1}\}\subseteq e_i$ for all $i=1, \cdots, l$.
The walk $W$ is called a \emph{path}, if
\begin{enumerate}
\item the vertices and the edges appeared in $W$ are all distinct;

\item for each $i=1, \ldots, l-1$, $ v_i \notin \cup_{j=i+1}^l e_j$.
\end{enumerate}
In this case, we say the path $W$ \emph{covers} the vertices $v_1,\ldots,v_l$ appeared in the path.
A single vertex is called a \emph{trivial path} with length $0$.
Two paths $W_1,W_2$ are called {\it disjoint} if any edge of $W_1$ contains no vertices covered by $W_2$, and vise versa.
A \emph{path cover} $\mathcal{C}$ of $G$ is a set of disjoint paths in $G$ which together cover all vertices of $G$.
The \emph{path cover number} of $G$, denoted by $\pc(G)$, is defined as the minimum cardinality of the path covers $\mathcal{C}$ of $G$.
If $m=2$, the above notions of path, path cover and path cover number are consistent with those of simple graphs respectively.

By the following lemma, the path cover number of $G$ can be defined equivalently as the
the maximum number of edges appeared in the path covers $\mathcal{C}$ of $G$.

\begin{lem}\label{pce}
Let $G$ be a connected uniform hypergraph on $n$ vertices and $\mathcal{C}=\{P_1, P_2, \cdots, P_t\}$ be a path cover of $G$.
Let $e(\mathcal{C})$ be the number of edges appeared in $\mathcal{C}$.
Then $e(\mathcal{C})=n-t$.
\end{lem}

\begin{proof}
For each path $P_i\colon v_{i_1}, e_{i_1}, v_{i_2}, \cdots, v_{i_{n_i}}, e_{i_{n_i}}, v_{i_{n_i+1}}$ containing at least one edge,
we construct a simple path (as simple graph) by linking a simple edge between $v_{i_k}$ and $v_{i_{k+1}}$ for $k=1,\ldots,n_i$, and get a simple path $\widetilde{P_i}$ .
Clearly, $P_i$ and $\widetilde{P_i}$ have the same number of edges.
By definition, $\widetilde{P_1}, \widetilde{P_2}, \cdots, \widetilde{P_t}$ are disjoint to each other, and their union contain all the vertices of $G$.
So,
\[
e(\mathcal{C})=\sum_{i=1}^t e(P_i)=\sum_{i=1}^t e(\widetilde{P_i})=n-t,
\]
where $e(P)$ denotes the number of edges in $P$.
\end{proof}

\begin{exm} Let $G$ be a $3$-uniform hypergraph in Figure \ref{pc}.
There exists a path cover consisting of a path covering the vertices $1,2,\ldots,8$, a path covering the vertices $9,10$, and $7$ trivial paths which are $11, 12, \ldots,17$ respectively. The path cover number of $G$ is $9$.
\end{exm}

\begin{figure}[htbp]
\includegraphics[scale=0.6]{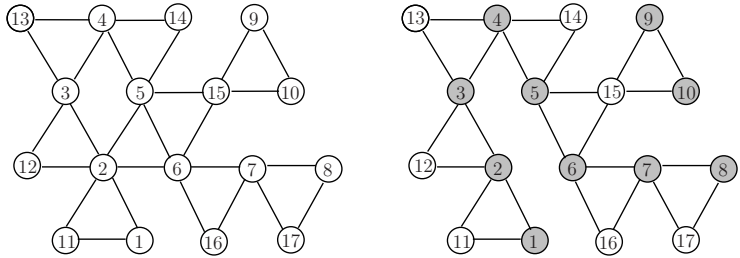}
\caption{\small A $3$-uniform hypergraph listed in left side, and the its path cover is listed in right side
}\label{pc}
\end{figure}

\begin{thm}\label{path cover}
Let $G=(V, E)$ be a connected $m$-uniform hypergraph with path cover number $pc(G)$.
Then $s(G) \le m^{{\small \pc(G)}-1}$, $\gamma(G)\le (\pc(G)-1)\cl(m)$.
\end{thm}

\begin{proof} Let $B_G$ be the incidence matrix of $G$.
Let $\mathcal{C}$ be a path cover of $G$ with minimum cardinality $\pc(G)$.
By relabeling the vertices and edges of the paths in $\mathcal{C}$,
we can get a square submatrix $\hat{B}$ of $B_G$ with rows indexed by the edges appeared in $\mathcal{C}$, which is a upper triangular matrix with $1$'s on the main diagonal.
So $B_G$ has at least $e(\mathcal{C})$ unit invariant divisors.
So, by Corollary \ref{stru-unit} and Lemma \ref{pce},
\[s(G) \le m^{n-1-e(\mathcal{C})}=m^{{\small \pc(G)}-1}, \gamma(G) \le (n-1-e(\mathcal{C}))\cl(m) =(\pc(G)-1)\cl(m),\] where $n$ is the number of vertices of $G$.
\end{proof}

Recall a \emph{matching} of a uniform hypergraph $G=(V, E)$ on $n$ vertices is a set of independent edges of $G$,
and the \emph{matching number} of $G$ is the maximum cardinality of matchings of $G$, denoted by $\mu(G)$.
Let $\mathcal{M}$ be  a matching of $G$ with cardinality $\mu(G)$.
Then $\mathcal{M}$ together with some singletons consists a path cover $\mathcal{C}$ of $G$.
By Lemma \ref{pce}, $$\mu(G)=n-|\mathcal{C}| \le n-\pc(G),$$
 with equality if and only if the edge set $E$ consists of disjoint edges;
otherwise there exists one edge outside $\mathcal{M}$ intersecting with edges of $\mathcal{M}$,
which yields a path cover with one more edges than  $\mathcal{C}$; a contradiction to Lemma \ref{pce}.
So, if $G$ is connected with more than one edge, then $\mu(G) \le n-\pc(G)-1$, and hence $\pc(G) \le n-\mu(G)-1$.
Similarly, if letting $d(G)$ be the maximum length of the paths in $G$, then $\pc(G) \le n-d(G)$.
So by Theorem \ref{path cover}, we get the following results directly.

\begin{cor}\label{match}
Let $G$ be a connected $m$-uniform connected hypergraph on $n$ vertices with matching number $\mu(G)$, which contains more than one edge.
Then $s(G) \le m^{n-\mu(G)-2}$, $\gamma(G)\le (n-\mu(G)-2)\cl(m)$.
\end{cor}

\begin{cor}\label{maxp}
Let $G$ be a connected $m$-uniform connected hypergraph on $n$ vertices, and let $d(G)$ be the maximum length of the paths in $G$.
Then $s(G) \le m^{n-d(G)-1}$, $\gamma(G)\le (n-d(G)-1)\cl(m)$.
\end{cor}

\begin{exm}
We will illustrate the above bounds for the stabilizing index and stabilizing dimension are all tight.
We only consider the stabilizing dimension here. The stabilizing index can be discussed similarly.
Let $P_n^m$ be the hyperpath as in Example \ref{linepath}.
We find that $\gamma(P_n^m)=(n-1)(m-2)\cl(m)$, $\pc(P_n^m)=(n-1)(m-2)+1$ and $d(P_n^m)=n-1$.
So, $$\gamma(P_n^m)=(\pc(P_n^m)-1)\cl(m)=[n+(n-1)(m-2)-d(P_n^m)-1]\cl(m),$$ implying the upper bounds in Theorem \ref{path cover} and Corollary \ref{maxp}
 hold as equalities for $P_n^m$.


Consider the generalized squid $S(m,t)$ in Example \ref{gsquid}.
The edges $e_1,\ldots, e_t$ consists of a matching of $S(m,t)$,
and $\mu(S(m,t))=t$ as $S(m,t)$ has $tm+m-t$ vertices and the maximum cardinality of matchings of $H$ is at most $\lfloor \frac{tm+m-t}{m} \rfloor=t$.
So, $$\gamma(S(m,t))=(t+1)(m-2)\cl(m)=(tm+m-t-\mu(S(m,t))-2)\cl(m),$$
implying the upper bound in Corollary \ref{match} holds as equality for $S(m,t)$.
\end{exm}

For the result in this subsection, if an $m$-uniform hypergraph $G$ has a prime order, i.e. $m$ is prime, then $\cl(m)=1$.
So we can get more brief formulas for the above upper bounds.

\section{Remark}
Let $A$ be a tensor of order $m$.
Chang et.al \cite{CPZ} define the geometric multiplicity of an eigenvalue $\la$ of $\A$ as the maximum number of linearly independent eigenvectors in $\V_{\la}$ though we know $\V_{\la}$ is not a linear space in general.
Li et.al \cite{LiYY} give a new definition of geometric multiplicity based on nonnegative irreducible tensors and two-dimensional nonnegative tensors.
Hu and Ye \cite{HuYe} define the geometric multiplicity of an eigenvalue $\la$ of $\A$ to be the dimension of $\V_\la$ as an affine variety, and try to establish a relationship between the algebraic multiplicity and geometric multiplicity of an eigenvalue of $\A$.

If $\A$ is a nonnegative combinatorial symmetric weakly irreducible tensor, the projective eigenvariety $\PV_{\rho(\A)}$
contains finite points, and has dimension $0$ as its irreducible subvarieties are singletons.
If considering the affine eigenvariety $\V_{\rho(\A)}$, it has dimension $1$.
By any one of the above definitions of geometric multiplicity, we cannot get useful information of the eigenvariety $\V_{\rho(\A)}$ or $\PV_{\rho(\A)}$.
In this paper, from the algebraic viewpoint, we prove that  $\PV_{\rho(\A)}$ admits a $\Z_m$-module, whose structure
can be characterized explicitly by solving the Smith normal of the incidence matrix of $\A$.
Furthermore, such algebraic structure is only determined by the combinatoric structure of $\A$, i.e. the support of $\A$.

\section*{appendix}


\begin{proof}[Proof of Lemma \ref{ev}.]
Suppose that $\A$ has order $m$ and dimension $n$.
As $\A$ is nonnegative and weakly irreducible, by Theorem \ref{PF1}(2),
there is a unique positive eigenvector $\x$ of $\A$ corresponding to the spectral radius $\rho(\A)$, up to a scalar.
Let $D$ be the diagonal matrix such that its diagonal entries $d_{ii}=x_i$ for $i \in [n]$.
Then by the eigenvector equation $\A \x^{m-1}=\rho(\A) \x^{[m-1]}$, we can verify that
\begin{equation}\label{stoB}\B:=\rho(\A)^{-1} D^{-(m-1)}\A D \end{equation}
is a stochastic tensor, i.e. $\B \mathbbm{1}^{m-1}=\mathbbm{1}$.
Note that $\B$ is also nonnegative and weakly irreducible, by  Theorem \ref{PF1}(2), $\rho(\B)=1$, with $\mathbbm{1}$ as the unique positive eigenvector.

We first assert that if there exists a nonzero nonnegative vector $\z$ such that
\begin{equation}\label{ineB}\B \z^{m-1} \ge \z^{[m-1]},\end{equation}
 then $\z = \alpha \mathbbm{1}$ for some positive $\alpha$, i.e. $\B \z^{m-1} = \z^{[m-1]}$.
Let $\delta=\max\{z_i: i \in [n]\} >0$, and let $I=\{i \in [n]: z_i=\delta\}$.
If $I=[n]$, the assertion follows.
Otherwise, let $J=[n]\backslash I$.
By the weakly irreducibility of $\B$,
there exists an index $\bar{i}_1 \in I$ and an index $\bar{i}_j \in J$ such that $\B$ has an entry $b_{\bar{i}_1 \cdots \bar{i}_j \cdots \bar{i}_m} >0$,
where $2 \le j \le m$.
So,
\begin{eqnarray*}
(\B z^{m-1})_{\bar{i}_1} & = & b_{\bar{i}_1 \cdots \bar{i}_j \cdots \bar{i}_m} z_{\bar{i}_1} \cdots z_{\bar{i}_j} \cdots z_{\bar{i}_m}+
               \sum_{i_j \ne \bar{i}_j, j=2,\ldots,m} b_{\bar{i}_1 i_2 \cdots i_m} z_{\bar{i}_1}z_{i_2} \cdots z_{i_m}  \\
               & < & \left(b_{\bar{i}_1 i_2 \cdots i_m} + \sum_{i_j \ne \bar{i}_j, j=2,\ldots,m} b_{\bar{i}_1 i_2 \cdots i_m}\right) \delta^{m-1}= \delta^{m-1}=z_{\bar{i}_1}^{m-1}.
\end{eqnarray*}
But this contradicts to the inequality (\ref{ineB}).

Now suppose that $\y$ is an eigenvector of $\A$ corresponding to an eigenvalue $\la$ with $|\la|=\rho(\A)$.
Then from the eigenvector equation $\A \y^{m-1}=\la \y^{[m-1]}$, we have $$\A |\y|^{m-1} \ge \rho(\A) |\y|^{[m-1]}.$$
Substituting $\A$ for $\B$ in (\ref{stoB}), we have
$$ (D^{(m-1)}\B D^{-1}) |\y|^{m-1} \ge  |\y|^{[m-1]},$$
and hence
$$\B (D^{-1}|\y|)^{m-1} \ge  (D^{-1}|\y|)^{[m-1]}.$$
By the above assertion, $D^{-1}|\y|=\alpha \mathbbm{1}$ for some positive $\alpha$, implying that $|\y|=\alpha \x$, and $\A |\y|^{m-1} = \rho(\A) |\y|^{[m-1]}$.
\end{proof}

\begin{proof}[Proof of Theorem \ref{PF2}.]
(1) Let $\y$ be an eigenvector of $\B$ corresponding to an eigenvalue $\beta$.
Then
$$ |\beta| |\y|^{[m-1]} \le |\B||\y|^{m-1} \le \A |\y|^{m-1}.$$
Then by Theorem 5.3 of \cite{YY1},
$$ \rho(\A)=\max_{\x \gneq 0}\min_{x_i>0} \frac{(\A \x^{m-1})_i}{x_i^{m-1}} \ge \min_{|y_i|>0} \frac{(\A |\y|^{m-1})_i}{|y_i|^{m-1}} \ge |\beta|.$$

(2) If $\beta=\rho(\A)e^{\i \theta}$, then $\A |\y|^{m-1} \ge \rho(\A)|\y|^{[m-1]}$.
So, by what we have discussed in the last part of the proof of Lemma \ref{ev}, $\A |\y|^{m-1} = \rho(\A)|\y|^{[m-1]}$,
and $|\y|$ is the unique positive eigenvector of $\A$ corresponding to $\rho(\A)$.
We also have $|\B||\y|^{m-1}=\rho(\A)|\y|^{[m-1]}$.
So $(\A - |\B|)|\y|^{m-1}=0$, and then $\A =|\B|$ as $\A \ge |\B|$ and $|\y|$ is positive.

Define $D=D_{\y}$ as in (\ref{Dy}).
Then $\y=D|\y|$.
By the eigenvector equation $\B \y^{m-1}=\rho(\A)e^{\i \theta} y^{[m-1]}$, we have
$\B (D|\y|)^{m-1}=\rho(\A)e^{\i \theta}(D|\y|)^{[m-1]}$.
So
$$ (e^{-\i \theta}D^{-(m-1)}\B D) |\y|^{m-1}=\rho(\A)|\y|^{[m-1]},$$
and $$ (\A - e^{-\i \theta}D^{-(m-1)}\B D) |\y|^{m-1}=0.$$
Noting that $ |e^{-\i \theta}D^{-(m-1)}\B D|=|\B|=\A$,  we have $\A = e^{-\i \theta}D^{-(m-1)}\B D$.
\end{proof}

\begin{proof}[Proof of Theorem \ref{specsymm}.]
Suppose that the $k$ eigenvalues of $\A$ with modulus $\rho(\A)$ are $\rho(\A)e^{\i \theta_j}$, $j=0,1,\ldots,k-1$, where $\theta_0=0$ and $\theta_1$ is one with
smallest positive argument among all $\theta_j$ for $j=1,2,\ldots,k-1$.
By Theorem \ref{PF2}(2), there exists a invertible diagonal matrix $D$ such that
$$ \A =e^{-\i \theta_1} D^{-(m-1)} \A D.$$
So $\A$ is diagonal similar to $e^{-\i \theta_1}\A$, which implies they have the same spectra.
Then $\A$ has eigenvalues $\rho(A)e^{\i t\theta_1}$ for any positive integer $t$.
But $\A$ only has $k$ eigenvalues of  modulus $\rho(\A)$, so $ \theta_1=\frac{2 \pi }{k}$, and those $k$ eigenvalues are
$\rho(A)e^{\i\frac{2 \pi j}{k}}$, $j=0, 1, \ldots, k-1$.
The remaining part is easily seen by Theorem \ref{PF2}(2) as $ \rho(A)e^{\i\frac{2 \pi }{k}}$ is an eigenvalue of $\A$,
and $\A$ is diagonal similar to $e^{-\i \frac{2 \pi j}{k}}\A$.
\end{proof}


\end{document}